\documentclass[12pt]{article}
\usepackage{amsmath,amsfonts,amsthm,hyperref}
\newtheorem{theorem}{Theorem}
\newtheorem{lemma}[theorem]{Lemma}
\newtheorem{corollary}[theorem]{Corollary}

\def\fl#1{\left\lfloor#1\right\rfloor}
\def\rf#1{\left\lceil#1\right\rceil}
\def\ep{{\mathbf{e}}_p}

\begin{document}

\title{On digit patterns in expansions of rational numbers with prime denominator}

\author{
\sc Igor E.~Shparlinski \\
Department of Computing, Macquarie University \\
Sydney, NSW 2109, Australia \\
\tt igor.shparlinski@mq.edu.au
\and
\sc  Wolfgang Steiner \\
LIAFA, CNRS, Universit{\'e} Paris Diderot -- Paris~7 \\
Case 7014, 75205 Paris Cedex 13, France \\
\tt steiner@liafa.univ-paris-diderot.fr}

\maketitle

\begin{abstract} 
We show that, for any fixed $\varepsilon>0$ and almost all primes~$p$, the $g$-ary expansion of any fraction~$m/p$ with $\gcd(m,p) =1$ contains almost all $g$-ary strings of length $k < (5/24-\varepsilon) \log_g p$. 
This complements a result of J.~Bourgain, S.~V.~Konyagin, and I.~E.~Shparlinski that asserts that, for almost all 
primes, all $g$-ary strings of length $k < (41/504-\varepsilon) \log_g p$ occur in the $g$-ary expansion of~$m/p$. 
\end{abstract}

\section{Introduction}

Let us fix some integer $g\ge 2$. 
It is well-known that if $\gcd(n,gm)=1$ then the $g$-ary expansion of  the 
rational fractions~$m/n$ is purely periodic with period~$t_n$, which is independent of~$m$ and equals 
the multiplicative order of $g$ modulo~$n$, see~\cite{Kor}.
In the series of works~\cite{BKS1,KoSh1,Kor}, the distribution of digit patterns in such expansions has been studied. 
In particular, for positive integers~$k$ and $m<n$   with $\gcd(n,gm)=1$, we denote by $T_{m,n}(k)$ the number of distinct $g$-ary strings $(d_1,\ldots, d_k) \in \{0,1,\ldots,g-1\}^k$ that occur among the first~$t_n$ 
trings $(\delta_r, \ldots, \delta_{r+k-1})$, $r=1, \ldots, t_n$, from the $g$-ary expansion 
\begin{equation}
\label{eq:expan}
\frac{m}{n} = \sum_{r=1}^\infty \delta_r g^{-r}, \quad \delta_r \in \{0,1,\ldots,g-1\}.
\end{equation}

Motivated by applications to pseudorandom number generators, see~\cite{BBS}, we are interested in describing the conditions under which $T_{m,n}(k)$ is close to its trivial upper bound 
$$
T_{m,n}(k)  \le \min\{t_n,g^k\}.
$$
Since $t_n \le n$, it is clear that only values $k \le \lceil \log_g n\rceil$ are of interest.
It has been shown in \cite[Theorem~11.1]{KoSh1} that, for any fixed $\varepsilon > 0$ and for almost all primes~$p$ (that is, for all but $o(x/\log x)$ primes $p\le x$), we have $T_{m,p}(k) = g^k$, provided that $k \le (3/37-\varepsilon) \log_g p$.
The coefficient $3/37$ has been increased up to $41/504$ in \cite[Corollary~8]{BKS1}. 
Here we show that, for almost all primes~$p$, we have $T_{m,p}(k) = (1+o(1)) g^k$ for much larger string lengths~$k$.

\begin{theorem}
\label{thm:period} 
For any fixed $\varepsilon > 0$, for almost all primes~$p$, we have 
$$
T_{m,p}(k) = (1+o(1)) g^k
$$
as $p \to \infty$, provided that $k \le (5/24-\varepsilon) \log_g p$.
\end{theorem}

Our arguments depend on the reduction of the problem to the study of intersections of intervals and multiplicative groups modulo~$p$ generated by~$g$, that has been established in~\cite{KoSh1}.
In turn, the question about the intersections of intervals and subgroups in residue rings has been studied in a number of works~\cite{BKS1,BKS2,KoSh1}. 
In particular, the results of \cite[Corollary~8]{BKS1} and \cite[Theorem~11.1]{KoSh1} are based on estimates of the length of the longest interval that is not hit by a subgroup of the multplicative group~$\mathbb{F}_p^*$ of the field~$\mathbb{F}_p$ of $p$ elements.
To prove Theorem~\ref{thm:period}, we use the results and ideas of~\cite{BKS1} to estimate the total number of intervals of a given length that do not intersect a given subgroup of~$\mathbb{F}_p^*$.

\section{Multiplicative Orders} 
\label{sec:Order}
We recall the following well-known implication of the classical result of~\cite{ErdMur}.

\begin{lemma}
\label{lem:Order}
For almost all primes~$p$, the multiplicative 
order $t$ of $g$ modulo $p$ satisfies $t > p^{1/2}$. 
\end{lemma}

\section{Bounds of Some Exponential Sums} 
\label{sec:ExpSum}

Let $p$ be prime and let $\mathcal{G} \subseteq \mathbb{F}_p^*$ be a subgroup of order~$t$, 
where $\mathbb{F}_p$ is a finite field of~$p$ elements. 

We denote
$$
\ep(z) = \exp(2 \pi i z/p)
$$
and define exponential sums
$$
S_{\lambda}(p;\mathcal{G}) = \sum_{v \in \mathcal{G}} \ep(\lambda v).
$$

Using \cite[Lemma~3]{HBK} (see also~\cite[Lemma~3.3]{KoSh1}) if $t < p^{2/3}$, and the well known bounds
$$
\left| S_{\lambda}(p;\mathcal{G})\right| \le p^{1/2} \quad \mbox{and} \quad \sum_{\lambda \in  \mathbb{F}_p^*}\left| S_{\lambda}(p;\mathcal{G})\right|^2 \le pt
$$
(see \cite[Equations~(3.4) and~(3.15)]{KoSh1}) if $t \ge p^{2/3}$, we derive:

\begin{lemma}
\label{lem:Bound SG4}
For any prime $p$ and a subgroup $\mathcal{G} \subseteq \mathbb{F}_p^*$  of order~$t$, we have
$$
\sum_{\lambda \in  \mathbb{F}_p^*} \left|S_{\lambda}(p;\mathcal{G})\right|^4 \ll p t^{5/2}.
$$
\end{lemma}

\section{Intervals Avoiding Subgroups} 

As before, let $p$ be prime and let $\mathcal{G} \subseteq \mathbb{F}_p^*$ be a subgroup of 
order~$t$. 

Let $\mathcal{U}(p;\mathcal{G},H)$ be the set of $u \in \mathbb{F}_p$ such the congruence
$$
v \equiv  u + x  \pmod p, \qquad v \in \mathcal{G}, \  0 \le x < H, 
$$
has no solution. 

\begin{lemma}
\label{lem:H and Exp}
Assume that $\mathcal{G}$ is of order $t>p^{1/2}$. 
Then, for any fixed integer $\nu \ge 1$, we have 
\begin{multline*}
\# \mathcal{U}(p;\mathcal{G},H) \le p^{2-1/4(\nu+1)+o(1)} H^{-1/2} t^{-5/4+(2\nu+1)/4\nu(\nu+1)} \\
+ p^{5/2-1/2\nu+o(1)} H^{-1} t^{-5/4+1/2\nu}.
\end{multline*}
\end{lemma}

\begin{proof} 
Let us fix some $\varepsilon> 0$. We put
$$
s = \rf{ \frac{3}{2} (1 + \varepsilon^{-1})}, \qquad  h = \rf{p^{1+\varepsilon}/H}, \qquad Z = \rf{H/s}.
$$
We can assume that $h < p/2$, as otherwise the bound is trivial (for example, it follows immediately from the bound of Heath-Brown and Konyagin \cite[Theorem~1]{HBK}). 
Obviously 
\begin{equation}
\label{eq:U W}
\mathcal{U}(p;\mathcal{G},H) \subseteq \mathcal{W}_s(p;\mathcal{G},Z),
\end{equation}
where $\mathcal{W}_s(p;\mathcal{G},Z)$ is the set of $u \in \mathbb{F}_p$ such the congruence
\begin{equation}
\label{eq:s-fold}
v \equiv u +  x_1 + \ldots + x_s  \pmod p, \quad v \in \mathcal{G}, \quad 0 \le x_1, \ldots, x_s < Z,
\end{equation}
has no solution.

For the number $Q_s(p;\mathcal{G},Z,u)$ of solutions to the congruence~\eqref{eq:s-fold}, exactly as in the proof of~\cite[Lemma~7.1]{KoSh1}, we obtain 
$$
Q_s(p;\mathcal{G},Z,u)  = \frac{1}{p} \sum_{|a|<p/2} \ep(- a u) \left( \sum_{0 \le x < Z} \ep(a x) \right)^{s} S_{a}(p;\mathcal{G}).
$$
where the sums $S_{a}(p;\mathcal{G})$ are defined in Section~\ref{sec:ExpSum}.

Separating the term $t Z^{s} p^{-1}$ corresponding to $a = 0$ and summing over all $u \in \mathcal{W}_s(p;\mathcal{G},Z)$ yields
$$
0 = \sum_{u \in \mathcal{W}_s(p;\mathcal{G},Z)} Q_s(p;\mathcal{G},Z,u) \ge \frac{t W Z^{s}}{p} - \frac{\sigma}{p},
$$
where 
$$
W = \# \mathcal{W}_s(p;\mathcal{G},Z)
$$ 
and 
$$
\sigma  = \sum_{1\le |a| < p/2}  \left|\sum_{u \in \mathcal{W}_s(p;\mathcal{G},Z)}  \ep(a u)\right| \, 
 \left| \sum_{0 \le x < Z}
\ep(a x) \right|^{s} \, \left| S_{a}(p;\mathcal{G}) \right|.
$$

Using the Cauchy inequality, and then the orthogonality relation for exponential functions, we obtain
\begin{align*}
\sigma^2  & \le \sum_{1\le |a| < p/2} \left| \sum_{u \in \mathcal{W}_s(p;\mathcal{G},Z)}  \ep(a u)\right|^2 \sum_{1\le |a| < p/2} \left| \sum_{0 \le x < Z} \ep(a x) \right|^{2s}\, \left| S_{a}(p;\mathcal{G})\right|^2 \\
& \le p W  \sum_{1\le |a| < p/2} \left| \sum_{0 \le x < Z} \ep(a x) \right|^{2s}\, \left| S_{a}(p;\mathcal{G})\right|^2.
\end{align*}
Hence
\begin{equation}
\label{eq:W Sigma}
W \le \frac{p}{t^2Z^{2s}} \Sigma, 
\end{equation}
where 
$$
\Sigma = \sum_{1\le |a| < p/2} \left| \sum_{0 \le x < Z} \ep(a x) \right|^{2s}\, \left| S_{a}(p;\mathcal{G})\right|^2.
$$
Following the idea of the proof of ~\cite[Lemma~7.1]{KoSh1}, we write
\begin{equation}
\label{eq:Sigma12}
\Sigma = \Sigma_1 + \Sigma_2,
\end{equation}
where 
\begin{align*}
\Sigma_1 & = \sum_{1 \le |a| \le h}  \left| \sum_{0 \le x < Z} \ep(a x) \right|^{2s} \, 
\left| S_{a}(p;\mathcal{G})\right|^2, \\
\Sigma_2 & = \sum_{h < |a| < p/2} \left| \sum_{0 \le x < Z}  \ep(a x) \right|^{2s} \, \left| S_{a}(p;\mathcal{G})\right|^2.
\end{align*}

For $1 \le | a| \le h$, we use the trivial estimate 
$$
\left| \sum_{0 \le x < Z} \ep(a x) \right| \le Z
$$
and derive
\begin{align*}
\Sigma_1 & \le Z^{2s} \sum_{1 \le |a| \le h} \left| S_{a}(p;\mathcal{G})\right|^2  = \frac{Z^{2s}}{t} \sum_{1 \le |a| \le h} 
\sum_{w\in \mathcal{G}} \left| S_{aw}(p;\mathcal{G})\right|^2 \\
& = \frac{Z^{2s}}{t} \sum_{\lambda \in \mathbb{F}_p^*}  M_{\lambda}(p;\mathcal{G},h) \, \left| S_{\lambda}(p;\mathcal{G})\right|^2,
\end{align*}
where $M_{\lambda}(p;\mathcal{G},h)$ denotes the number of solutions to the congruence
$$
\lambda \equiv a w \pmod p, \quad 1\le |a| \le h, \quad  w \in \mathcal{G}.
$$
Hence, by the Cauchy inequality
$$
\Sigma_1 \le \frac{Z^{2s}}{t} \left( \sum_{\lambda \in \mathbb{F}_p^*}  M_{\lambda}(p;\mathcal{G},h)^2 \right)^{1/2}
\left( \sum_{\lambda \in \mathbb{F}_p^*} \left|S_{\lambda}(p;\mathcal{G})\right|^4 \right)^{1/2}.
$$
As in \cite[Section~3.3]{BKS1}, we have
$$
\sum_{\lambda \in  \mathbb{F}_p^*}  M_{\lambda}(p;\mathcal{G},h)^2 \le t N(p;\mathcal{G},h),
$$
where $N(p;\mathcal{G},h)$ is the number of solutions of the congruence
$$
u x \equiv y \pmod p, \quad 0<|x|, |y|\le h, \quad u \in \mathcal{G}.
$$
Therefore,
\begin{equation}
\label{eq:Sigma N S4}
\Sigma_1 \le \frac{Z^{2s}}{t^{1/2}} N(p;\mathcal{G},h)^{1/2}
\left(\sum_{\lambda \in  \mathbb{F}_p^*}\left|S_{\lambda}(p;\mathcal{G})\right|^4\right)^{1/2}.
\end{equation}
It is shown in \cite[Theorem~1]{BKS1} that if $t \ge p^{1/2}$ then for any fixed integer~$\nu$ and any positive number~$h$, we have
\begin{equation}
\label{eq:Bound N}
N(p;\mathcal{G},h) \le h t^{(2\nu+1)/2\nu(\nu+1)} p^{-1/2(\nu+1)+o(1)} + h^2 t^{1/\nu} p^{-1/\nu+o(1)}.
\end{equation}

Therefore, using Lemma~\ref{lem:Bound SG4} and the bound~\eqref{eq:Bound N}
we derive from~\eqref{eq:Sigma N S4} that
\begin{equation}
\label{eq:Sigma1}
\Sigma_1\le  p^{1/2}  t^{3/4} Z^{2s} \left( h^{1/2} t^{(2\nu+1)/4\nu(\nu+1)} p^{-1/4(\nu+1)+o(1)} + h t^{1/2\nu} p^{-1/2\nu+o(1)}\right).
\end{equation}

If $h < |a|  < p/2$, then we use the bound
$$
\sum_{0 \le x < Z} \ep( a x) \ll \frac{p}{|a|},
$$
see \cite[Bound~(8.6)]{IwKow}.
{}From the  trivial bound
$$
\left| S_{a}(p;\mathcal{G})\right| \le t,
$$
recalling the choice of~$h$, we obtain
$$
\Sigma_2  \ll \sum_{h<|a|<p/2}  \left( \frac{p}{|a|} \right)^{2s} t^2 \ll t^2 \frac{p^{2s}}{h^{2s-1}} \ll t^2  \frac{Z^{2s}h}{p^{2s\varepsilon}} \le \frac{Z^{2s}p^3}{p^{2s\varepsilon}} \ll Z^{2s},
$$
as $2s\varepsilon > 3$ for our choice of~$s$. 
Thus the bound on~$\Sigma_2$ is dominated by the bound~\eqref{eq:Sigma1} on~$\Sigma_1$. 
Using~\eqref{eq:W Sigma} and~\eqref{eq:Sigma12}, we obtain
$$
W \le  p^{3/2} t^{-5/4} \big(h^{1/2} t^{(2\nu+1)/4\nu(\nu+1)} p^{-1/4(\nu+1)+o(1)} + h t^{1/2\nu} p^{-1/2\nu+o(1)}\big).
$$
Recalling~\eqref{eq:U W}, the choice of $h$ and that  $\varepsilon$ is arbitrary, 
after simple calculations, we obtain the result. 
\end{proof}

\begin{corollary}
\label{cor:Full}
Assume that $\mathcal{G}$ is of order $t > p^{1/2}$. 
Then for any $\varepsilon> 0$ and
$$
H \ge p^{19/24+\varepsilon}
$$
we have 
$$
\# \mathcal{U}(p;\mathcal{G},H)  = o(p).
$$
\end{corollary}

\begin{proof} 
Since $t > p^{1/2}$, we have, for any fixed integer $\nu \ge 1$,
$$
\# \mathcal{U}(p;\mathcal{G},H) \le p^{11/8+1/8\nu(\nu+1)+o(1)} H^{-1/2} + p^{15/8-1/4\nu+o(1)} H^{-1} .
$$
Taking $\nu =2$ or $\nu=3$, we conclude the proof. 
\end{proof}

\section{Proof of Theorem~\ref{thm:period}} 

By Lemma~\ref{lem:Order} it is enough to consider prime $p$ for which
the multiplicative 
order $t$ of $g$ modulo $p$ satisfies $t > p^{1/2}$. 

We now take a positive integer $k \le (5/24-\varepsilon) \log_g p$ and  consider the intervals $\big[\frac{D}{g^k} , \frac{D+1}{g^k}\big)$.
As in the proof of \cite[Theorem~11.1]{KoSh1}, we observe that, for any integer $\ell \ge 0$ and  any $g$-ary string $(d_1,\ldots,d_k)$, we have $\delta_{\ell+i} = d_i$, $i = 1, \ldots, k$, if and only if 
$$
\frac{m g^\ell}{p} - \fl{ \frac{m g^\ell}{p}} \in \left[\frac{D}{g^k} , \frac{D+1}{g^k}\right) ,
$$
where $D = d_1 g^{k-1} + d_2 g^{k-2} + \cdots + d_k$ and the~$\delta_r$, $r =1,2,\ldots$, are defined by~\eqref{eq:expan} with $n=p$. 
Thus, if a string  $(d_1, \ldots, d_k)$ is not present in the $g$-ary expansion of~$m/p$, then each interval $[u,u+H)$ with 
$$
u = \rf{ \frac{D}{g^k}p}, \ldots, \fl{\frac{D+1/2}{g^k}p} \quad \mbox{and} \quad H = \fl{ \frac{1}{2g^k}p }$$
contains no element of the conjugacy class $m\mathcal{G}_p$ of the group $\mathcal{G}_p$ generated by $g$ modulo~$p$.
Clearly, different strings $(d_1, \ldots, d_k)$ correspond to different intervals of the values of~$u$, and each of them contains
$$
\rf{  \frac{D+1/2}{g^k}p} - \rf{  \frac{D}{g^k}p } \gg \frac{p}{g^k}
$$
values of~$u$. 
Therefore, the number of missing strings $(d_1, \ldots, d_k)$ satisfies 
$$
g^k - T_{m,p}(k) \ll \frac{g^k}{p} \# \mathcal{U}(p;\mathcal{G}_p, H).
$$
Since $g^k \le p^{5/24-\varepsilon}$, we infer from Corollary~\ref{cor:Full} that $\# \mathcal{U}(p;\mathcal{G}_p,H) = o(p)$, which proves Theorem~\ref{thm:period}.

\section{Composite Denominators}

It is quite likely that  one can also study $T_{m,n}(k)$ for  almost all 
composite $n$ by supplementing the ideas of this work with 
those of~\cite{BKPS} (to get an analogue of Lemma~\ref{lem:Bound SG4}) 
and also using the result of~\cite{KurPom} that gives an analogue of
Lemma~\ref{lem:Order}.

 \section*{Acknowledgements}

The second author wishes to express his heartfelt thanks to the members of the Department of Computing of the Macquarie University for their hospitality during his stay as a visiting academic.

During the preparation of this work the first author was supported in part by 
the  Australian Research Council  Grant~DP1092835.

\end{document}